\documentclass[a4paper,11pt,reqno]{amsart}
\usepackage{lmodern,amsmath,amssymb,mathtools,amsthm}
\usepackage{enumerate}
\usepackage{graphics,graphicx}
\usepackage{bigints}
\usepackage{geometry}
\usepackage{wasysym}
\usepackage{cleveref}
\usepackage{floatrow,morefloats}
\usepackage{epsfig}
\usepackage{epstopdf}
\usepackage{caption,subcaption}%
\numberwithin{equation}{section}

\newtheorem{theorem}{Theorem}[section]
\newtheorem{lemma}{Lemma}[section]

\newtheorem{corollary}{Corollary}[section]
\newtheorem{remark}{Remark}[section]
\newtheorem{definition}{Definition}[section]

\theoremstyle{definition}
\newtheorem{example}{Example}

\begin{document}
\bibliographystyle{amsplain}
\title{{{
Stability of Cobweb economic model involving Hilfer fractional derivative}}}
\author{Divya Raghavan*
}
\address{
Department of Mathematics,
Indian Institute of Technology Roorkee,
Uttarakhand, India
}
\email{divyar@ma.iitr.ac.in,madhanginathan@gmail.com}
\author{
N. Sukavanam
}
\address{
Department of  Mathematics  \\
Indian Institute of Technology Roorkee,
Uttarakhand,  India
}
\email{n.sukavanam@ma.iitr.ac.in}
\bigskip
\begin{abstract}
This paper evaluates the solution of cobweb models when there is a Hilfer fractional derivative in the demand and supply function. Particular cases when Hilfer derivative reducing to Caputo and Riemann-Liouville derivative are discussed. Subsequently, the solution of the cobweb model with Riemann-Liouville derivative is derived instantly. Two numerical examples interpreting the crafted theory is deliberated with illustrations and comparisons. Conclusions based on the graphical illustrations is outlined in detail. These illustrations highlight the advantage of the Hilfer fractional derivative over the other two fractional derivatives.
\end{abstract}
\maketitle
\pagestyle{myheadings}
\markboth
{Divya Raghavan and N. Sukavanam}
{Stability of economic Cobweb model involving Hilfer fractional derivative}
\section{Prologue}
There is always a bonding between economics and mathematics. The study of models in economics using math principles to carry out quantifiable tests is termed as Mathematical economics. Economics basically deals with the production of goods, stocking the goods, distribution of stock and utilization of goods by costumer. The equilibrium of price between demand and supply for non-storable goods are generally described using cobweb model. Primarily, the term cobweb model was used typically in the field of economics. But later due to the similarities between the cobweb model and dynamical systems in terms of stability conditions, the contribution of mathematicians in the cobweb theory got strengthened. Despite the fact that cobweb theory was developed in the 18th Century, Kaldar, an economist in his research article \cite{Cobweb-name} gave the specific name "Cobweb theorem", wherein he gave the theoretical interpretation of the cobweb theory. However, Ezekiel \cite{Cobweb-theorem} reformulated the classical theory to a neo-classical economic theory that is related to the price analysis of statistics. Ezekiel also claimed that not all commodity cycles are cobwebs and discussed cobweb theory in three cases - when there are continuous fluctuations, divergent fluctuations and convergent fluctuations of supply and demand. The limitations of the cobweb theory can also be referred in this paper. In this regard, the monograph of Gandolfo \cite{Cobweb-book}, is one of the main reference book for many researchers in the field of Mathematics and Economics who work on dynamic models. With the view of the fact that the supply responds to price with one period lag and the demand relies upon the current price, the linear equation in terms of these functions was given by Gandolfo \cite{Cobweb-book} as
\begin{align}
\label{def:Cobweb-basic}
\left\{
  \begin{array}{ll}
  D_{t}=a+bp_{t}\\
S_{t}=a_{1}+b_{1}p_{t-1}.
\end{array}
\right.
\end{align}
Here time $t$ is non-instantaneous and not continuous but it is for a fixed period of time. $p_{t}$ is the market price, $D_{t}$ is the demand at time $t$ and $S_{t}$ is the supply at time $t$. Another constraint along with (\ref{def:Cobweb-basic}), is that for each period, the supply exactly meet the demand, that is $D_{t}=S_{t}$. Using this constraint, the general solution of the system (\ref{def:Cobweb-basic}) is given by
\begin{align*}
p_{t}=(p_{0}-p_{e})\left(\dfrac{b_{1}}{b}\right)^{t}+p_{e}.
\end{align*}
 Here $p_{e}$ is the equilibrium value of the price and $p_{0}$ is the initial price. The stability condition is achieved when the market price attains the equilibrium value, that is, when $\left(\dfrac{b_{1}}{b}\right)<1$. Gandolfo in his monograph not only discussed the above conventional model but also proposed a general cobweb model, explained its stability conditions with examples. In this sequence, Chiarella \cite{Cobweb-Chiarella} discussed the cobweb model for the case where the supply is a nonlinear function and also the case when the equilibrium is not stable. The Cobweb model for such locally unstable system resulted in either period doubling pattern or a chaotic pattern. Further, Li and Xi in \cite{Cobweb-Stability-nonlinear}, studied the stability analysis when both demand and supply functions are nonlinear. The authors perturbed the standard cobweb model, so that the demand and supply are not always equal. With $r>0$, as the price adjustor, the cobweb model assumed by Li and Xi is:
 \begin{align*}
\left\{
  \begin{array}{ll}
  D_{t}=D(p_{t})\\
S_{t}=S(p_{t-1})\\
p_{t}=p_{t-1}+r(D_{t}-S_{t}).
\end{array}
\right.
\end{align*}
Here $D_{t}$ refers to the consumer's demand and $S_{t}$ and $p_{t}$ indicates the supply by the producer and market price respectively. Li and Xi derived the equation for dynamic movement $p_{t}=f(p_{t-1})$ by simple substitution from the model with $p_{t}=p_{t-1}+r(D(p_{t-1})-S(p_{t-1}))$ and $f(p)=p+r(D(p)-S(p))$. By defining the balanced price or the equilibrium price as $p_{e}$, Li and Xi studied conditions for the convergent of $p_{t}$ to $p_{e}$.

In spite of the fact that the cobweb model is attributed based on the time lag between demand and supply, the study of cobweb model for delay system is an interesting as well as dispensable. For such a nonlinear and time delay  cobweb system, Matsumoto and Szidarovszky\cite{Cobweb-delay-1} studied its asymptotic behaviour. By considering three different models-viz- model with single delay, model with two delays and two market model with two delays. Their article exhibits a complete study of stability of those models. In this row it is worth to mention the work of Gori et al. \cite{Cobweb-delay-2} that explains how time delays are responsible for different outcomes depending in disequilibrium dynamics of actual price and equlibrium dynamics of expected price. To have a deep analysis on nonlinear cobweb theory, the work of  Hommes\cite{Cobweb-Hommes-1, Cobweb-Hommes-2, Cobweb-Hommes-3} can be referred.

Fractional Calculus finds its application in all possible directions in more recent years. Especially in the field of economics, fractional calculus  provides a significant contribution in the past as well as present. Article by Tejado et al. \cite{Cobweb-frac-eco} claims that fractional models renders better performance in the study of GDP(Gross domestic growth).  Undoubtedly, the study of cobweb model for fractional order system will help to study cobweb theory in a different perspective. Recently, Bohner and Hatipo\u{g}lu \cite{Cobweb-frac-eco} studied the linear cobweb model with fractional derivative. The fractional derivative Bohner and Hatipo\u{g}lu considered was conformable fractional derivative, which was defined by Khalil et al. \cite{Def-Conformable}. The cobweb model with conformable fractional derivative was given as
 \begin{align*}
\left\{
  \begin{array}{ll}
 D_{t}=a+b(p(t)+T_{\alpha}(p)(t))\\
S_{t}=a_{1}+b_{1}(p(t))\\
D(t)=S(t).
\end{array}
\right.
\end{align*}
Here, $T_{\alpha}(p)(t)$ that represents the conformable fractional derivative of order $\alpha$ is defined as
 \begin{align*}
 T_{\alpha}(f)(t)=
\left\{
  \begin{array}{ll}
\displaystyle\lim_{\epsilon\rightarrow 0} \dfrac{f(t+\epsilon t^{1-\alpha})-f(t)}{\epsilon}, \enspace t>0\\
\displaystyle\lim_{\tau \rightarrow 0^{+}}T_{\alpha}(f)(\tau), \enspace t=0
\end{array}
\right.
\end{align*}
Here $\alpha$ $\in$ $(0,1]$, $f:[0,\infty) \rightarrow \mathbb{R}$. For $\alpha=1$ it coincides with the typical first order derivative  model. The recent work which is the motivation for this work is by Chen et al.\cite{Cobweb-Caputo}, where they discuss the Caputo fractional derivative in the demand function and supply function as separate cases; by suitably representing the model as 
 \begin{align*}
\left\{
  \begin{array}{ll}
 D_{t}=a+b(p(t)+{}^{C}D_{0}^{\alpha}(p)(t))\\
S_{t}=a_{1}+b_{1}(p(t))\\
D(t)=S(t).
\end{array}
\right.
\end{align*}
 and
  \begin{align*}
\left\{
  \begin{array}{ll}
 D_{t}=a+bp(t)\\
S_{t}=a_{1}+b_{1}(p(t)+d\enspace{}^{C}D_{0}^{\alpha}(p)(t))\\
D(t)=S(t).
\end{array}
\right.
\end{align*}
where $0<\alpha\leq 1$, $a,b,a_{1},b_{1},d \in \mathbb{R},  b\neq 0, b\neq b_{1}$. Before proceeding to the further description of the problem, it is important that the basic definitions and basic theory to be mentioned in brief. Section 2 gives an outline of such basic information. Section 3 is assigned for the solution of two models where the Hilfer fractional derivative is in the demand function and when it is in the supply function. It is always necessary to compare results numerically and graphically for a better interpretation. Section 4 contains all such comparative study which reflects the motivation of this paper.
\section{Essential notions}
\begin{definition}{\cite{Podlubny-book}}
The integral
\begin{align*}
I^{\mu}_{t}g(t)=\dfrac{1}{\Gamma(\mu)}\int^{t}_{0}(t-s)^{\mu-1}g(s)ds, \enspace \enspace \mu> 0,
\end{align*}
is called the Riemann-Liouville fractional integral of order $\mu$, where $\Gamma(\cdot)$ is the well known gamma function.
\end{definition}
\begin{definition}{\cite{Podlubny-book}}
The Riemann-Liouville derivative of order $\mu >0$ for a function $g:[0,\infty)\rightarrow\mathbb{R}$ is defined by
\begin{align*}
^{L}D^{\mu}_{0+}g(t)=\dfrac{1}{\Gamma(n-\mu)}\left(\dfrac{d}{dt}\right)^{(n)}\int^{t}_{0}(t-s)^{n-\mu-1}g(s)ds,\enspace t>0,\enspace n-1\leq \mu <n,
\end{align*}
where $n=[\mu]+1$ and $[\mu]$ denotes the integral part of number $\mu$.
\end{definition}
\begin{definition}{\cite{Podlubny-book}}
The Caputo derivative of order $\mu >0$ for a function $g:[0,\infty)\rightarrow\mathbb{R}$ is defined by
\begin{align*}
^{C}D^{\mu}_{0+}g(t)=\dfrac{1}{\Gamma(n-\mu)}\int^{t}_{0}(t-s)^{n-\mu-1}g(s)^{(n)}ds, \enspace t>0,\enspace n-1\leq \mu <n,
\end{align*}
where $\Gamma(\cdot)$ is the Gamma function.
\end{definition}
\begin{definition}{\cite{Hilfer-remark}}
The Hilfer fractional derivative of order $0< \mu <1$ and type $0\leq \nu \leq 1$ of function $g(t)$ is defined by
\begin{align}
D^{\mu,\nu}_{0+}g(t)=I_{0+}^{\nu(1-\mu)}DI_{0+}^{(1-\nu)(1-\mu)}g(t)
\end{align}
where $D:=\dfrac{d}{dt}$.
\end{definition}
\begin{remark}\cite{Hilfer-remark}
\begin{enumerate}[\rm(i)]
\item
The Hilfer fractional derivative $D^{\mu,\nu}_{0+}$  is considered as an merger between the Riemann-Liouville derivative $^{L}D^{\mu}_{0+}$ and the Caputo derivative $^{C}D^{\mu}_{0+}$, since
\begin{align*}
D_{0+}^{\mu,\nu}=
\left\{
  \begin{array}{ll}
   DI_{0+}^{1-\mu}={ }^{L}D_{0+}^{\mu},\enspace \nu=0\\
   I_{0+}^{1-\mu}D={ }^{C}D^{\mu}_{0+},\enspace \nu=1
  \end{array}
\right.
\end{align*}
that is, when $\nu=0$, the Hilfer corresponds to the classical Riemann-Liouville fractional derivative and when $\nu=1$, the Hilfer fractional
derivative corresponds to the classical Caputo derivative.
\item
The parameter $\lambda$ satisfies\\
$\gamma=\mu+\nu-\mu \nu, \enspace 0<\gamma\leq 1, \enspace \gamma\geq\mu, \enspace \gamma>\nu$.
\end{enumerate}
\end{remark}
\begin{definition}{\cite{Podlubny-book}}
The Mittag-Leffler function with one parameter $\mu$  and two parameters $\mu$ and $\lambda$ are given respectively by,
\begin{align*}
  E_{\mu}(z)=\sum_{0}^{\infty}\dfrac{z^{k}}{\Gamma(\mu k+1)} \enspace \mbox{and}\enspace E_{\mu,\gamma}(z)=\sum_{0}^{\infty}\dfrac{z^{k}}{\Gamma(\mu k+\gamma)},\enspace \mu>0,\enspace\gamma>0, \enspace z \in \mathfrak{C}.
\end{align*}
and for $\gamma=1$, $E_{\mu}(z)=E_{\mu,1}(z)$.
\end{definition}
\begin{definition}\cite{Hilfer-book}
The Laplace transform of Hilfer derivative is given by
\begin{align}
\label{eqn:Hilfer-Laplace}
\mathcal{L}\{D_{0+}^{\mu,\nu}g(t)\}=s^{\mu}\mathcal{L}\{g(t)\}-s^{\nu(\mu-1)}\left({}_{0^{+}}I_{t}^{(1-\nu)(1-\mu)}g\right)(0^{+})
\end{align}
where $\left({}_{0^{+}}I_{t}^{(1-\nu)(1-\mu)}g\right)(0^{+})$ is the fractional integral.
\end{definition}
\begin{lemma}\cite{Kilbas-book}
\label{lem:Cobweb:Lap-Mittag}
For $\mu>0$ and $0\leq\nu\leq1$, it is true that,
\begin{align*}
\mathcal{L}\{t^{\gamma-1}E_{\mu,\gamma}(\lambda t^{\mu})\}(s)=\dfrac{ s^{\nu(\mu-1)}}{s^{\mu}-\lambda}
\end{align*}
and
\begin{align*}
\mathcal{L}\{1-E_{\mu}(\lambda t^{\mu})\}(s)=\dfrac{\lambda}{s(s^{\mu}-\lambda)}
\end{align*}
where $\lambda$ is a constant. Also,
\begin{align*}
E_{\mu,\gamma}(\lambda (t-a)^{\mu})=\lambda E_{\mu,\gamma}(\lambda (t-a)^{\mu}).
\end{align*}
\end{lemma}
\begin{lemma}\cite{Mittag}
\label{lem:one-mittag}
Let $0<\mu<2$ and $\dfrac{\pi \mu}{2}<\theta <min \{\pi, \mu\pi\}$, then for an arbitrary positive integer $h$,  the asymptotic  formula is given as,
\begin{align*}
E_{\mu}(z)=\displaystyle{ -\sum_{k=1}^{h}\dfrac{z^{-k}}{\Gamma (1-\mu k)}+\mathcal{O}(|z|^{-1-h}), \enspace |z|\rightarrow \infty,\enspace \theta \leq |arg (z)|\leq \pi}.
\end{align*}
\end{lemma}
\begin{lemma}\cite{Mittag-two-para}
\label{lem:two-mittag}
Let $0<\mu<2$, $0<\gamma<2$, $\mu\gamma<2$ and $\dfrac{\pi \mu\gamma}{2}<\theta <min \{\pi, \mu\gamma\pi\}$, then for an arbitrary positive integer $h$,  the asymptotic  formula is given as,
\begin{align*}
E_{\mu,\gamma}(z)=\displaystyle{ -\sum_{k=1}^{h}\dfrac{z^{-k}}{\Gamma (\gamma-\mu k)}+\mathcal{O}(|z|^{-1-h}), \enspace |z|\rightarrow \infty,\enspace \theta \leq |arg (z)|\leq \pi}.
\end{align*}
\end{lemma}
\begin{remark}
\label{rem:Mittag-zero}
As $z\rightarrow \infty $ on both the Lemma \ref{lem:one-mittag} and Lemma \ref{lem:two-mittag} reduces to zero. That is,
\begin{align*}
\displaystyle{\lim_{z\rightarrow \infty} E_{\mu}(-z)=\lim_{z\rightarrow \infty}\left( -\sum_{k=1}^{h}\dfrac{(-z)^{-k}}{\Gamma (1-\mu k)}+\mathcal{O}(|-z|^{-1-h})\right) \rightarrow 0}
\end{align*}
and
\begin{align*}
\displaystyle{\lim_{z\rightarrow \infty} E_{\mu\gamma}(-z)=\lim_{z\rightarrow \infty}\left( -\sum_{k=1}^{h}\dfrac{(-z)^{-k}}{\Gamma (\gamma-\mu k)}+\mathcal{O}(|-z|^{-1-h})\right) \rightarrow 0}
\end{align*}
\end{remark}
\section{Cobweb model with Hilfer derivative}
\subsection{Hilfer derivative in the demand function}
The proposed basic cobweb model with Hilfer fractional derivative in the demand function is given by
 \begin{align}
 \label{eqn:Cobweb-Hilfer-demand}
\left\{
  \begin{array}{ll}
 D_{t}=\alpha+\beta(p(t)+D_{0+}^{\mu,\nu}(p)(t))\\
S_{t}=\alpha_{1}+\beta_{1}(p(t))\\
D(t)=S(t).
\end{array}
\right.
\end{align}
where $D_{0+}^{\mu,\nu}(p)(t)$ is the Hilfer fractional derivative with $ \alpha, \beta, \alpha_{1}, \beta_{1} \in \mathbb{R}, \beta\neq 0, \beta\neq\beta_{1}$ and $0<\mu\leq 1, 0\leq\nu\leq 1$.
\begin{theorem}
\label{thm:Cobweb-Demand}
The unique solution of the cobweb model {\rm{(\ref{eqn:Cobweb-Hilfer-demand})}} with Hilfer dreivative in the demand function is given by
\begin{align}
\label{eqn:Cobweb-demand-solution}
p(t)=Ct^{\gamma-1}E_{\mu,\gamma}(\lambda t^{\mu})-\dfrac{\xi}{\lambda}+\dfrac{\xi}{\lambda}E_{\mu}(\lambda t^{\mu}).
\end{align}
Here $\lambda=\dfrac{\beta_{1}-\beta}{\beta},\enspace \xi= \dfrac{\alpha_{1}-\alpha}{\beta}, \enspace C\in \mathbb{R}$.
\end{theorem}
\begin{proof}
The solution can be deduced using the Laplace transform method. Solving the demand and supply equations with the equilibrium condition in (\ref{eqn:Cobweb-Hilfer-demand}) results in,
\begin{align*}
\alpha +\beta(p(t)+D_{0+}^{\mu,\nu}p(t))=\alpha_{1}+\beta_{1}p(t).
\end{align*}
On futher simplification, the above equation can be written as
\begin{align}
\label{eqn:equlibrium-demand}
D_{0+}^{\mu,\nu}p(t)=\lambda p(t)+\xi,
\end{align}
where $\lambda=\dfrac{\beta_{1}-\beta}{\beta},\enspace \xi= \dfrac{\alpha_{1}-\alpha}{\beta}$.
Taking Laplace transform on both sides of the above equation leads to,
\begin{align*}
\mathcal{L}\{D_{0+}^{\mu,\nu}(p)(t)\}(s)=\lambda \mathcal{L}\{p(t)\}(s)+\mathcal{L}\{\xi\}.
\end{align*}
Directly applying the Laplace transform (\ref{eqn:Hilfer-Laplace}) of Hilfer fractional derivative results in
\begin{align*}
s^{\mu}\mathcal{L}\{p(t)\}(s)-s^{\nu(\mu-1)}\left({}_{0^{+}}I^{(1-\nu)(1-\mu)}_{t}p\right)(0^{+})=\lambda \mathcal{L}\{p(t)\}(s)+\dfrac{\xi}{s}.
\end{align*}
Separating the like terms gives,
\begin{align*}
\Rightarrow s^{\mu}\mathcal{L}\{p(t)\}(s)-\lambda \mathcal{L}\{p(t)\}(s)=&s^{\nu(\mu-1)}\left({}_{0^{+}}I^{(1-\nu)(1-\mu)}_{t}p\right)(0^{+})+\dfrac{\xi}{s}\\
\Rightarrow \mathcal{L}\{p(t)\}(s)=&\dfrac{s^{\nu(\mu-1)}\left({}_{0^{+}}I^{(1-\nu)(1-\mu)}_{t}p\right)(0^{+})}{s^{\mu}-\lambda}+\dfrac{\xi}{s(s^{\mu}-\lambda)}.
\end{align*}
Further on substituting $\left({}_{0^{+}}I^{(1-\nu)(1-\mu)}_{t}p\right)(0^{+})=C$, the above equation scales down to
\begin{align*}
\mathcal{L}\{p(t)\}(s)=\dfrac{C s^{\nu(\mu-1)}}{s^{\mu}-\lambda}+\dfrac{\xi}{s(s^{\mu}-\lambda)}.
\end{align*}
The following equation is the outcome based on Lemma \ref{lem:Cobweb:Lap-Mittag}.
\begin{align}
\label{eqn:Cobweb-demand-Laplace}
\mathcal{L}\{p(t)\}(s)=C\mathcal{L}\{t^{\gamma-1}E_{\mu,\gamma}(\lambda t^{\mu})\}(s)-\dfrac{\xi}{\lambda}\mathcal{L}\{1-E_{\mu}(\lambda t^{\mu})\}(s).
\end{align}
Taking inverse Laplace transform of (\ref{eqn:Cobweb-demand-Laplace}) gives
\begin{align*}
p(t)=C t^{\gamma-1}E_{\mu,\gamma}(\lambda t^{\mu})-\dfrac{\xi}{\lambda}+\dfrac{\xi}{\lambda}E_{\mu}(\lambda t^{\mu}),
\end{align*}
which is the solution of the cobweb model with Hilfer fractional derivative.

Now, it is to be proved conversely that if such $p$ exists, then it satisfies the equation (\ref{eqn:equlibrium-demand}). Given $
p(t)=Ct^{\gamma-1}E_{\mu,\gamma}(\lambda t^{\mu})-\dfrac{\xi}{\lambda}+\dfrac{\xi}{\lambda}E_{\mu}(\lambda t^{\mu}).$ Taking Hilfer derivative on both sides and using Lemma \ref{lem:Cobweb:Lap-Mittag} results in
\begin{align*}
D_{0+}^{\mu,\nu}p(t)=&\lambda Ct^{\gamma-1}E_{\mu,\gamma}(\lambda t^{\mu})+\xi E_{\mu}(\lambda t^{\mu})\\
\Rightarrow D_{0+}^{\mu,\nu}p(t)=&\lambda\left(Ct^{\gamma-1}E_{\mu,\gamma}(\lambda t^{\mu})+\dfrac{\xi}{\lambda}E_{\mu}(\lambda t^{\mu})\right)\\
\Rightarrow D_{0+}^{\mu,\nu}p(t)=&\lambda\left(Ct^{\gamma-1}E_{\mu,\gamma}(\lambda t^{\mu})+\dfrac{\xi}{\lambda}E_{\mu}(\lambda t^{\mu})-\dfrac{\xi}{\lambda}\right)+\xi\\
\Rightarrow D_{0+}^{\mu,\nu}p(t)=&p(t)\lambda+\xi
\end{align*}
Therefore the proposed solution $p(t)$ satisfies the given cobweb model.
\end{proof}
Stability conditions for which the solution converges to the equilibrium value $p_{e}$ is given in the subsequent theorem.
\begin{theorem}
\label{thm:demand-stability}
Assume $\dfrac{\beta_{1}}{\beta}<1$. Then the solution of {\rm{(\ref{eqn:Cobweb-Hilfer-demand})}} converges to the equilibrium value $p_{e}$. The equilibrium value $p_{e}$ according to Gandolfo {\rm{\cite{Cobweb-book}}} is given as $\dfrac{\alpha_{1}-\alpha}{\beta-\beta_{1}}$ .
\end{theorem}
\begin{proof}
In Theorem \ref{eqn:Cobweb-Hilfer-demand}, it is assumed that $\lambda=\dfrac{\beta_{1}-\beta}{\beta}$. Also given that $\dfrac{\beta_{1}}{\beta}<1\Rightarrow \lambda<0$, and hence for $0<\mu<1$, $\lambda t^{\mu}\rightarrow -\infty$ as $t\rightarrow +\infty$. With reference to the Remark \ref{rem:Mittag-zero}, when $ t\rightarrow \infty$,
\begin{align}
\label{eqn:Cobweb-Mittag-limiting}
\displaystyle \lim_{t\rightarrow \infty}E_{\mu}(\lambda t^{\mu})=0.
\end{align}
 Now apply limit on both sides of the solution (\ref{eqn:Cobweb-demand-solution}). The limiting value converges to the equilibrium value $p_{e}$ which is given by Gandolfo \cite{Cobweb-book}. That is,
 \begin{align*}
 \displaystyle \lim_{t\rightarrow +\infty}p(t)=\displaystyle \lim_{t\rightarrow +\infty}\left[Ct^{\gamma-1}E_{\mu,\gamma}(\lambda t^{\mu})-\dfrac{\xi}{\lambda}+\dfrac{\xi}{\lambda}E_{\mu}(\lambda t^{\mu})\right].
 \end{align*}
Now two cases arise:-
\begin{enumerate}[(1)]
\item When $\gamma=1$, then from (\ref{eqn:Cobweb-Mittag-limiting})
\begin{align*}
\displaystyle \lim_{t\rightarrow +\infty}Ct^{0}E_{\mu,1}(\lambda t^{\mu})\rightarrow 0.
\end{align*}
Hence $\displaystyle \lim_{t\rightarrow +\infty}p(t)=-\dfrac{\xi}{\lambda}$.
\item
When $0<\gamma<1$, then from Remark \ref{rem:Mittag-zero}
\begin{align*}
\displaystyle \lim_{t\rightarrow +\infty}Ct^{\gamma-1}E_{\mu,\gamma}(\lambda t^{\mu})\rightarrow 0.
\end{align*}
Hence $\displaystyle \lim_{t\rightarrow +\infty}p(t)=-\dfrac{\xi}{\lambda}$.
\end{enumerate}
It can be concluded  thereby, $\displaystyle \lim_{t\rightarrow +\infty}p(t)=-\dfrac{\xi}{\lambda}=\dfrac{\alpha_{1}-\alpha}{\beta-\beta_{1}}=p_{e}$.
\end{proof}
The following corollary is significant by means of comparative analysis. Considering the fact that Caputo derivative and Riemann derivative are a particular cases of Hilfer derivative, it is noteworthy to discuss how the solution of cobweb model with Caputo and Riemann derivative can be derived easily from the solution (\ref{eqn:Cobweb-demand-solution}).
\begin{lemma}{${}$}
\label{Cobweb-Demand-compare}
\begin{enumerate}[1.]
\item
When $\nu=1$, $\gamma$ takes the value $1$ as $\gamma=\mu+\nu-\mu\nu$. Substituting $gamma=1$ in {\rm{(\ref{eqn:Cobweb-demand-solution})}}, gives
\begin{align*}
p(t)=C_{0}E_{\mu,1}(\lambda t^{\mu})-\dfrac{\xi}{\lambda}+\dfrac{\xi}{\lambda}E_{\mu}(\lambda t^{\mu})\\
\Rightarrow p(t)=\left(C_{0}+\dfrac{\xi}{\lambda}\right)E_{\mu,1}(\lambda t^{\mu})-\dfrac{\xi}{\lambda},
\end{align*}
which is the solution of the cobweb model with Caputo derivative discussed by Chen et al {\rm{\cite{Cobweb-Caputo}}}, which is on further assuming $\mu=1$, reduces to the integer order solution given by Gandolfo {\rm{\cite{Cobweb-book}}}.
\item
When $\nu=0$, $\gamma$ takes the value $\mu$ and by substituting this value in the solution {\rm{(\ref{eqn:Cobweb-demand-solution})}}, the solution of the cobweb model with Riemann fractional derivative, can be deduced, viz,
\begin{align*}
p(t)=C_{1}t^{\mu-1}E_{\mu,\mu}(\lambda t^{\mu})-\dfrac{\xi}{\lambda}+\dfrac{\xi}{\lambda}E_{\mu}(\lambda t^{\mu})
\end{align*}
where $C_{1}=\left({}_{0}I^{1-\mu}_{t}p\right)(0^{+})$.
\item
It may be noted that, the case $\nu=0$ leads to the cobweb model with Riemann-Liouville fractional derivative is new in the literature.
\end{enumerate}
\end{lemma}
\subsection{Hilfer derivative in the supply function}
This section considers the cobweb model with Hilfer fractional derivative in the supply function of the form:
 \begin{align}
 \label{eqn:Cobweb-Hilfer-supply}
\left\{
  \begin{array}{ll}
 D_{t}=\alpha+\beta p(t)\\
S_{t}=\alpha_{1}+\beta_{1}(p(t)+\delta\cdot D_{0+}^{\mu,\nu}(p)(t))\\
D(t)=S(t).
\end{array}
\right.
\end{align}
where $D_{0+}^{\mu,\nu}(p)(t)$ is the Hilfer fractional derivative with $ \alpha, \beta, \alpha_{1}, \beta_{1},\delta \in \mathbb{R}, \beta\neq 0, \beta\neq\beta_{1}$ and $0<\mu\leq 1, 0\leq\nu\leq 1$.
\begin{theorem}
\label{thm:Cobweb-Supply}
The unique solution of the cobweb model {\rm{(\ref{eqn:Cobweb-Hilfer-supply})}} with Hilfer derivative in the supply function is given by
\begin{align}
\label{eqn:Cobweb-supply-solution}
p(t)=Ct^{\gamma-1}E_{\mu,\gamma}(\varrho t^{\mu})-\dfrac{\eta}{\varrho}+\dfrac{\eta}{\varrho}E_{\mu}(\varrho t^{\mu}).
\end{align}
Here $\varrho=\dfrac{\beta_{1}-\beta}{\delta\beta},\enspace \eta= \dfrac{\alpha_{1}-\alpha}{\delta\beta}, \enspace C\in \mathbb{R}$.
\end{theorem}
\begin{proof}
The proof of the theorem is similar to the proof of the Theorem \ref{thm:Cobweb-Demand}. Simplification of (\ref{eqn:Cobweb-Hilfer-supply}) results in
\begin{align}
\label{eqn:equlibrium-supply}
D_{0+}^{\mu,\nu}(p)(t)=\varrho p(t)+\eta.
\end{align}
Further, Laplace transform method is applied in the same way to get,
\begin{align*}
s^{\mu}\mathcal{L}\{p(t)\}(s)-s^{\nu(\mu-1)}\left({}_{0^{+}}I^{(1-\nu)(1-\mu)}_{t}p\right)(0^{+})=\varrho\mathcal{L}\{p(t)\}(s)+\dfrac{\eta}{s}
\end{align*}
where $\varrho=\dfrac{\beta_{1}-\beta}{\delta\beta},\enspace \eta= \dfrac{\alpha_{1}-\alpha}{\delta\beta}$.
Separating the like terms gives,
\begin{align*}
\Rightarrow \mathcal{L}\{p(t)\}(s)=\dfrac{s^{\nu(\mu-1)}\left({}_{0^{+}}I^{(1-\nu)(1-\mu)}_{t}p\right)(0^{+})}{s^{\mu}-\varrho}+\dfrac{\eta}{s(s^{\mu}-\varrho)}.
\end{align*}
On substituting $\left({}_{0^{+}}I^{(1-\nu)(1-\mu)}_{t}p\right)(0^{+})=C$ and using the Lemma \ref{lem:Cobweb:Lap-Mittag} gives,
\begin{align}
\label{eqn:Cobweb-supply-Laplace}
\mathcal{L}\{p(t)\}(s)=C\mathcal{L}\{t^{\gamma-1}E_{\mu,\gamma}(\varrho t^{\mu})\}(s)-\dfrac{\eta}{\varrho}\mathcal{L}\{1-E_{\mu}(\varrho t^{\mu})\}(s).
\end{align}
Taking inverse Laplace transform of (\ref{eqn:Cobweb-supply-Laplace}) gives
\begin{align*}
p(t)=C t^{\gamma-1}E_{\mu,\gamma}(\varrho t^{\mu})-\dfrac{\eta}{\varrho}+\dfrac{\eta}{\varrho}E_{\mu}(\varrho t^{\mu}),
\end{align*}
which is the solution of the cobweb model with Hilfer fractional derivative.

The converse can be proved in a similar manner as in Theorem \ref{thm:Cobweb-Demand} and we omit details. The proof is complete.
\end{proof}
Stability conditions for which the solution converges to the equilibrium value $p_{e}$ is given in the subsequent theorem.
\begin{theorem}
\label{thm:supply-stability}
Assume $\dfrac{\beta_{1}}{\beta}<1$. Then the solution of {\rm{(\ref{eqn:Cobweb-Hilfer-supply})}} converges to the equilibrium value $p_{e}$. The equilibrium value $p_{e}$ according to Gandolfo {\rm{\cite{Cobweb-book}}} is given as $\dfrac{\alpha_{1}-\alpha}{\beta-\beta_{1}}$ .
\end{theorem}
\begin{proof}
In Theorem \ref{eqn:Cobweb-Hilfer-supply}, it is assumed that $\varrho=\dfrac{\beta_{1}-\beta}{\beta}$. It is presumed in the theorem statement that $\dfrac{\beta_{1}}{\beta}<1\Rightarrow \varrho<0$, and hence for $0<\mu<1$, $\varrho t^{\mu}\rightarrow -\infty$ as $t\rightarrow +\infty$. With reference to the Remark \ref{rem:Mittag-zero}, when $ t\rightarrow \infty$,
\begin{align}
\label{eqn:Cobweb-Mittag-limiting-s}
\displaystyle \lim_{t\rightarrow \infty}E_{\mu}(\varrho t^{\mu})=0.
\end{align}
 Now apply limit on both sides of the solution (\ref{eqn:Cobweb-supply-solution}). The limiting value converges to the equilibrium value $p_{e}$ which is given by Gandolfo \cite{Cobweb-book}. That is,
 \begin{align*}
 \displaystyle \lim_{t\rightarrow +\infty}p(t)=\displaystyle \lim_{t\rightarrow +\infty}\left[Ct^{\gamma-1}E_{\mu,\gamma}(\varrho t^{\mu})-\dfrac{\eta}{\varrho}+\dfrac{\eta}{\varrho}E_{\mu}(\varrho t^{\mu})\right].
 \end{align*}
For all the values of $0<\gamma\leq 1$, the first and the last term of the solution (\ref{eqn:Cobweb-supply-solution}) vanishes as $t\rightarrow \infty$. Hence $\displaystyle \lim_{t\rightarrow +\infty}p(t)=-\dfrac{\eta}{\varrho}=\dfrac{\alpha_{1}-\alpha}{\beta-\beta_{1}}$.
It can be concluded  thereby, $\displaystyle \lim_{t\rightarrow +\infty}p(t)=p_{e}$.
\end{proof}
\begin{corollary}
\begin{enumerate}
\item
The results can be compared in the same way as in Lemma \ref{Cobweb-Demand-compare} and it is evident that the solution of the cobweb model with Caputo and Riemann-Liouville derivative in the demand and supply function are the particular case of the solution of the cobweb model with Hilfer derivative.
\item
The basic integer order cobweb model can be derived as a special case.
\end{enumerate}
\end{corollary}
\section{Graphical Illustration}
In this section, two examples that appeared in the paper \cite{Cobweb-Caputo} by Chen et al. are considered for comparative analysis. It is discussed case by case for detailed understanding.
\begin{example}
A basic cobweb model with Hilfer fractional derivative in the demand function is considered.
\begin{align*}
D(t)=&40-10(p(t)+D_{0+}^{\mu,\nu}p(t)),\\
S(t)=&2+9p(t),\\
D(t)=&S(t).
\end{align*}
Let $p_{0}=20$. The solution for various values of $\mu$ and $\nu$ has to be observed and the condition for the solution to converge to the equilibrium value has to be verified.
Comparing the given model with the basic model, it can be noticed that $\alpha=40$; $\beta=-10$; $\alpha_{1}=2$; $\beta_{1}=9$. $\lambda$ and $\xi$ can be found by  $\lambda=\dfrac{\beta_{1}-\beta}{\beta},\enspace \xi= \dfrac{\alpha_{1}-\alpha}{\beta}$. The given model satisfies the stability condition $\dfrac{\beta_{1}}{\beta}<1$ as given by the Theorem \ref{thm:demand-stability}. And $\lambda=-1.9$ and $\xi= 3.8$ and the equilibrium value is $p_{e}$= 2. Now various cases can be discussed to check how the trajectory of the solution converges to the equilibrium value.
\begin{enumerate}
\item[] Case 1:- When $\mu=0.1$; and the value of $\nu$ is varied from $0$ to $1$ with $0.2$ interpolation. The values are given in a tabular form as well as graphically.
  \begin{table}[H]
  \begin{center}
    \caption{Value of $p(t)$ for $\mu=0.1$ and various values of $\nu$}
    \label{tab:table1}
    \begin{tabular}{|c|c|c|c|c|c|} 
    \hline
     \textbf{t} & \textbf{100} & \textbf{1000}& \textbf{10000}& \textbf{$p_{e}$}\\
      \hline
      $\nu$=0    & 1.64985 & 1.69042 & 1.73129 &2\\
      $\nu$=0.2 & 3.00988 & 2.82498 & 2.66983 &2\\
      $\nu$=0.4 & 3.95831 & 3.6185   & 3.32796 &2\\
      $\nu$=0.6 & 5.08315 & 4.5619   & 4.11206 &2\\
      $\nu$=0.8 & 5.91011 & 5.258     & 4.69247 &2\\
      $\nu$=1    & 6.28175 & 5.57406 & 4.95832 &2\\
      \hline
      \end{tabular}
  \end{center}
\end{table}
Symbolically from Figure 1 it is evident that cobweb model with Riemann-Liouville derivative (for $\nu=0$ represented by dotted lines) in the demand function converges not exactly to its equilibrium value, but for higher values of $t$, it is getting closer. But the solution for the model with Caputo derivative(for $\nu=1$) in the demand function has huge variation even for higher values of $t$.
\begin{figure}[H]		
	\begin{subfigure}[b]{0.45\textwidth}
		\centering
		\includegraphics{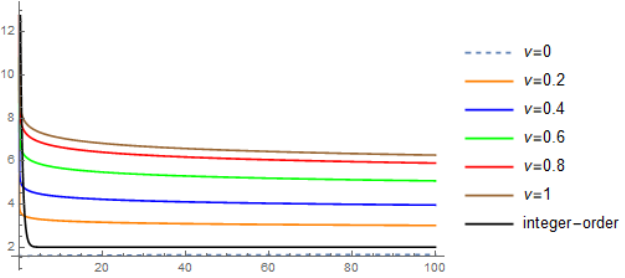}
		\caption{$\mu=0.1$ and $t$ ranges from $0-100$}
	\end{subfigure}
	\begin{subfigure}[b]{0.4\textwidth}
		\centering
		\includegraphics{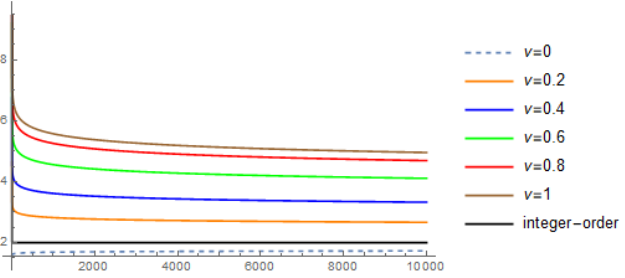}
		\caption{$\mu=0.1$ and $t$ ranges from $0-10000$}
	\end{subfigure}
 \caption{}
\end{figure}
\item[] Case 2:- When $\mu=0.5$ and the value of $\nu$ is varied from $0$ to $1$ with $0.2$ interpolation. The values are given in a tabular form as well as graphically.
  \begin{table}[H]
  \begin{center}
    \caption{Value of $p(t)$ for $\mu=0.5$ and various values of $\nu$}
    \label{tab:table1}
    \begin{tabular}{|c|c|c|c|c|c|} 
    \hline
     \textbf{t} & \textbf{100} & \textbf{1000}& \textbf{10000}& \textbf{$p_{e}$}\\
      \hline
      $\nu$=0    & 1.95826 &1.98299 & 1.99424  &2\\
      $\nu$=0.2 & 2.0818 & 2.02232  & 2.0067  &2\\
      $\nu$=0.4 & 2.20983 &2.06342 & 2.01975 &2\\
      $\nu$=0.6 & 2.33351 &2.10342& 2.03249  &2\\
      $\nu$=0.8 & 2.44412 &2.13948 & 2.044  &2\\
      $\nu$=1    & 2.5338 & 2.1690 &2.0534  &2\\
      \hline
      \end{tabular}
  \end{center}
\end{table}
Graphically from Figure 2 it is evident that cobweb model with Riemann-Liouville derivative (for $\nu=0$ represented by dotted lines) in the demand function attains stability much faster than the model with Caputo derivative(for $\nu=1$) in the demand function.
\begin{figure}[H]		
	\begin{subfigure}[b]{0.45\textwidth}
		\centering
		\includegraphics{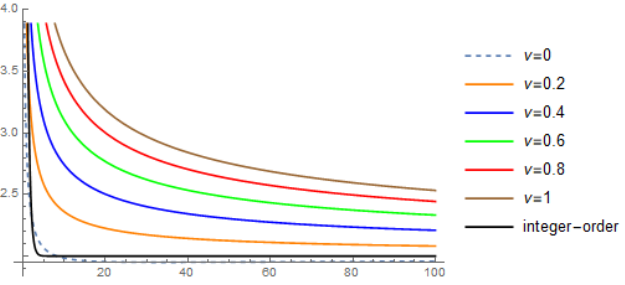}
		\caption{$\mu=0.5$ and $t$ ranges from $0-100$}
	\end{subfigure}
	\begin{subfigure}[b]{0.4\textwidth}
		\centering
		\includegraphics{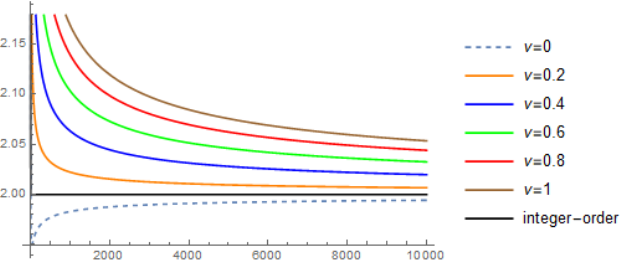}
		\caption{$\mu=0.5$ and $t$ ranges from $0-10000$}
	\end{subfigure}
 \caption{}
\end{figure}
\item[] Case 3:- When $\mu=0.9$; and the value of $\nu$ is varied from $0$ to $1$ with $0.2$ interpolation. The values are given in a tabular form as well as graphically.
  \begin{table}[h!]
  \begin{center}
    \caption{Value of $p(t)$ for $\mu=0.5$ and various values of $\nu$}
    \label{tab:table1}
    \begin{tabular}{|c|c|c|c|c|c|} 
    \hline
     \textbf{t} & \textbf{100} & \textbf{1000}& \textbf{10000}& \textbf{$p_{e}$}\\
      \hline
      $\nu$=0    & 1.99836 &1.99978 & 1.99997 &2\\
      $\nu$=0.2 & 2.0019  & 2.00022  & 2.00003  &2\\
      $\nu$=0.4 & 2.00545 &2.00067 & 2.00008 &2\\
      $\nu$=0.6 & 2.00899 &2.00111 & 2.00014  &2\\
      $\nu$=0.8 & 2.01251 &2.00155 & 2.0002 &2\\
      $\nu$=1    & 2.01601& 2.00199 &2.00025  &2\\
      \hline
      \end{tabular}
  \end{center}
\end{table}
The illustrative given in Figure 3-(A) outlines how Riemann and Caputo converges to the equlibrium value. Also of Figure 3-(B), clearly spells out that for higher values $\mu$ and lower values of $\nu$, the solution converges to the equlibrium value.
\begin{figure}[ht!]		
	\begin{subfigure}[b]{0.45\textwidth}
		\centering
		\includegraphics{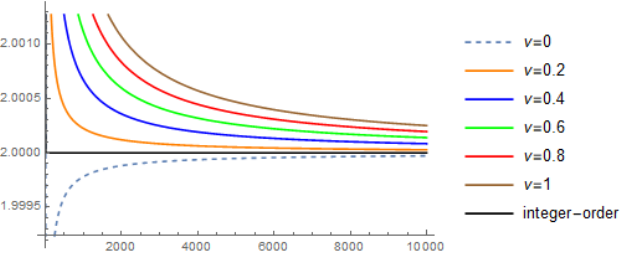}
		\caption{$\mu=0.9$ and $t$ ranges from $0-10000$}
	\end{subfigure}
	\begin{subfigure}[b]{0.4\textwidth}
		\centering
		\includegraphics{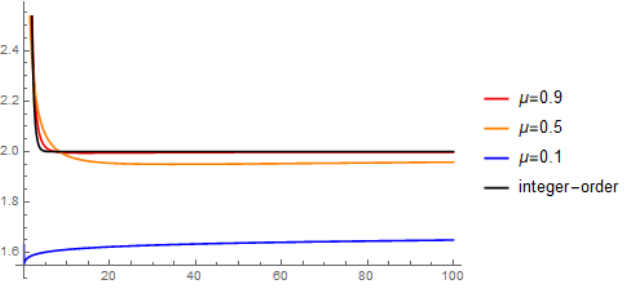}
		\caption{$\nu=0$ and various values of $\mu$}
	\end{subfigure}
 \caption{}
\end{figure}
\end{enumerate}
\end{example}
\begin{example}
A basic cobweb model with Hilfer fractional derivative in the supply function is considered.
\begin{align*}
D(t)=&80-4p(t),\\
S(t)=&-10+2(p(t)+3D_{0+}^{\mu,\nu}p(t)),\\
D(t)=&S(t).
\end{align*}
Let $p_{0}=20$. The solution for various values of $\mu$ and $\nu$ has to be observed and the condition for the solution to converge to the equilibrium value has to be verified.

Comparing the given model with the basic model, it can be noticed that $\alpha=80$; $\beta=-4$; $\alpha_{1}=-10$; $\beta_{1}=2$. It satisfies the stability condition $\dfrac{\beta_{1}}{\beta}<1$ given by the Theorem \ref{thm:supply-stability}. Moreover, $\varrho$ and $\eta$ can be found by  $\varrho=\dfrac{\beta_{1}-\beta}{\delta\beta},\enspace \eta= \dfrac{\alpha_{1}-\alpha}{\delta\beta}$. That is, $\varrho=-1$ and $\eta= 15$ and the equilibrium value is $p_{e}=15$. Now various cases can be discussed to check how the trajectory of the solution converges in models having Caputo and Riemann derivatives in the supply function and how they converge to the equilibrium value.
\begin{enumerate}
\item [] Case 1:- When $\mu=0.1$; and the value of $\nu$ is varied from $0$ to $1$ with $0.2$ interpolation. The values are given in a tabular form as well as graphically.
  \begin{table}[h!]
  \begin{center}
    \caption{Value of $p(t)$ for $\mu=0.1$ and various values of $\nu$}
    \label{tab:table1}
    \begin{tabular}{|c|c|c|c|c|c|} 
    \hline
     \textbf{t} & \textbf{100} & \textbf{1000}& \textbf{10000}& \textbf{$p_{e}$}\\
      \hline
      $\nu$=0    & 9.71582 & 10.4198 & 11.0785 &15\\
      $\nu$=0.2 & 11.8507 & 12.2544 & 12.6367 &15\\
      $\nu$=0.4 & 13.325 & 13.526   & 13.7206 &15\\
      $\nu$=0.6 & 15.0592 & 15.0266  & 15.0034 &15\\
      $\nu$=0.8 & 16.3184 &16.1214   &15.9434 &15\\
      $\nu$=1    & 16.8641 & 16.6027 & 16.3621 &15\\
      \hline
      \end{tabular}
  \end{center}
\end{table}
\item [] Case 2:- When $\mu=0.5$; and the value of $\nu$ is varied from $0$ to $1$ with $0.2$ interpolation. The values are given in a tabular form as well as graphically.
\begin{figure}[ht!]		
	\begin{subfigure}[b]{0.45\textwidth}
		\centering
		\includegraphics{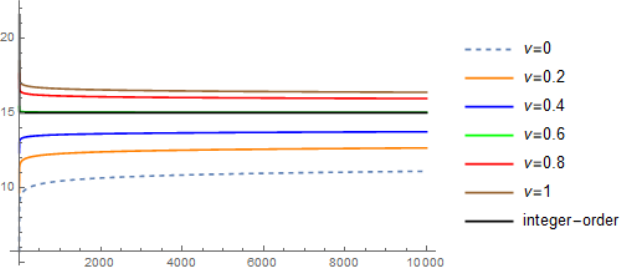}
		\caption{$\mu=0.1$ and $t$ ranges from $0-10000$}
	\end{subfigure}
	\begin{subfigure}[b]{0.4\textwidth}
		\centering
		\includegraphics{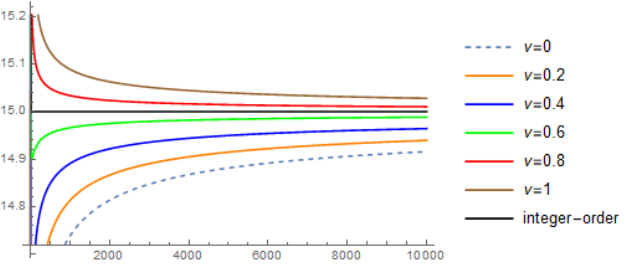}
		\caption{$\mu=0.5$ and $t$ ranges from $0-10000$}
	\end{subfigure}
 \caption{}
\end{figure}
Graphical representation shows that the both fractional order Riemann-Lioville (for $\nu=0$) and Caputo (for $\nu=1$) when compared to integer valued model, has a considerable deviation when $\mu=0.1$ and $\mu=0.5$. Also only for higher values of $t$, the fractional order trajectories are only close to the equilibrium trajectory.

\item [] Case 3:- When $\mu=0.9$; and the value of $\nu$ is varied from $0$ to $1$ with $0.2$ interpolation. The values are given in a tabular form as well as graphically.
  \begin{table}[ht!]
  \begin{center}
    \caption{Value of $p(t)$ for $\mu=0.5$ and various values of $\nu$}
    \label{tab:table1}
    \begin{tabular}{|c|c|c|c|c|c|} 
    \hline
     \textbf{t} & \textbf{100} & \textbf{1000}& \textbf{10000}& \textbf{$p_{e}$}\\
      \hline
      $\nu$=0    & 14.9747 &14.9967 & 14.9996 &15\\
      $\nu$=0.2 & 14.9816  & 14.9977 & 14.9997  &15\\
      $\nu$=0.4 & 14.9884 &14.9985 & 14.9998 &15\\
      $\nu$=0.6 & 14.9952 &14.9994 & 14.999 &15\\
      $\nu$=0.8 & 15.0019 &15.0002 & 15 &15\\
      $\nu$=1    & 15.0086&15.0011 &15.0001  &15\\
      \hline
      \end{tabular}
  \end{center}
\end{table}
The trajectory of the solution when ($\mu=0.9$) compared with the above two path, converges to the equilibrium trajectory much more closer. In fact, when $\mu=0.9$ and $\nu=0.8$, the solution of the given Hilfer fractional order model exactly coincides with the equilibrium value. This can be verified with of Figure 5-(A). The comparative analysis of the solution trajectory for various values of $\mu$ with $\nu=0.8$  can be observed in of Figure 5-(B)
\begin{figure}[ht!]		
	\begin{subfigure}[b]{0.45\textwidth}
		\centering
		\includegraphics{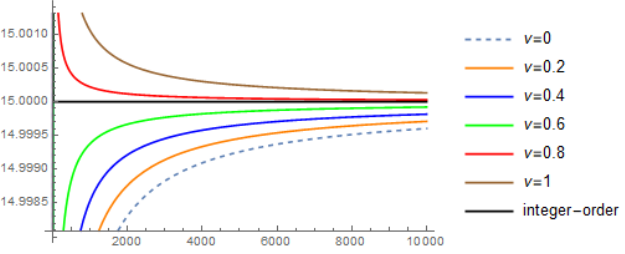}
		\caption{$\mu=0.9$ and $t$ ranges from $0-10000$}
	\end{subfigure}
	\begin{subfigure}[b]{0.4\textwidth}
		\centering
		\includegraphics{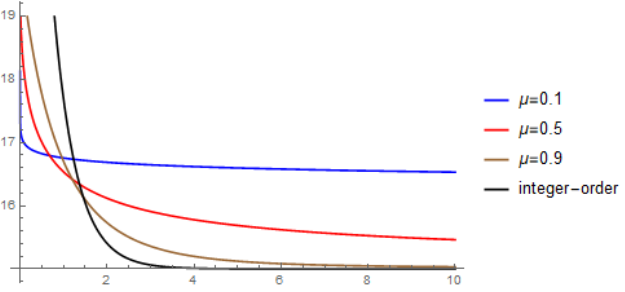}
		\caption{$\nu=0.8$ and various values of $\mu$}
	\end{subfigure}
 \caption{}
\end{figure}
\end{enumerate}
\end{example}
\section{Observation}
This paper provides an idea of how the fractional derivative values affects the trajectory of the solution. Even though Caputo and Riemman-Liouville derivatives are the conventional derivative used, it is always a perplexity of which derivative helps the solution attain stability faster and more closer to the equilibrium value. Solving for model having Hilfer derivative, shows the exact value of fractional order and exact derivative type, for which the system is stable. In the cobweb theory, this is among the few paper that deals with the fractional order and this is the only paper in the literature that considers Hilfer derivative in the cobweb model. This paper opens the scope of studying the various cobweb model with Hilfer fractional derivative including models with delay for future.

\end{document}